\documentclass[11pt,a4paper, singlecolumn]{article}
\usepackage[utf8]{inputenc}
\usepackage{amsmath}
\usepackage{amsfonts}
\usepackage{amssymb}
\usepackage{mathtools}
\usepackage{amsthm}
\usepackage{authblk}
\newtheorem{theorem}{Theorem}

\newtheorem{cor}{Corollary}

\newtheorem*{proof*}{Proof}

\allowdisplaybreaks
\begin{document}
\title{Study of sensitivity of parameters of Bernstein-Stancu operators}
\author[1, 3, b]{R.B. Gandhi}
\author[1, 2, a, *]{Vishnu Narayan Mishra}
\affil[1]{\scriptsize Department of Applied Mathematics \& Humanities, Sardar Vallabhbhai National Institute of Technology, Ichchhanath Mahadev-Dumas Road, Surat - 395 007, (Gujarat), India}
\affil[2]{\scriptsize L. 1627 Awadh Puri Colony Beniganj, Phase-III, Opposite - Industrial Training Institute (ITI), Ayodhya Main Road, Faizabad, Uttar Pradesh 224 001, India}
\affil[3]{\scriptsize Department of Mathematics, BVM Engineering College, Vallabh Vidyanagar - 388 120, (Gujarat), India}
\date{}

\maketitle
    
\let\thefootnote\relax\footnotetext{\scriptsize E-mails: vishnunarayanmishra@gmail.com$^{a}$; rajiv55in@yahoo.com$^{b}$  \\ \scriptsize $^{*}$Corresponding author}
\begin{abstract}
\noindent \scriptsize This paper is aimed at studying sensitivity of parameters $\alpha$ and $\beta$ appearing in the operators introduced by D.D. Stancu \cite{Sta} in 1969. Results are established on the behavior the nodes used in Bernstein-Stancu polynomials and the nodes used in Bernstein polynomials and graphical presentations of them are generated. Alternate proof of uniform convergence of Bernstein-Stancu operators and an upper bound estimation are derived. It is also established that the parameters $\alpha$ and $\beta$ in Bernstein-Stancu polynomials can be used to get better approximation at a point $x = \alpha/\beta$ in $[0,1]$ to the Bernstein polynomials. \\ 

\\ \\
\textit{Mathematical Subject Classification – MSC2010 : \medskip 41A10, 41A36}. 

\end{abstract}

\section{Introduction}
Let $f$ be a continuous function, $f: [0,1] \rightarrow \mathbb{R}$. For every natural number $n$, we denote by $B_n$, the Bernstein polynomial of degree $n$, defined as
\begin{equation}\label{Ber}
B_n(f;x)=\sum_{k=0}^{n} b_{n,k}(x)f\left(\dfrac{k}{n} \right),
\end{equation}
where
\begin{equation}\label{BerBas}
b_{n,k}(x)= \left(\begin{array}{c} n \\ k  \end{array} \right)x^k (1-x)^{n-k}, \quad k=0,1, \ldots, n.
\end{equation}
In 1969, D.D. Stancu \cite{Sta} introduced a linear positive operator involving two non-negative parameters $\alpha$ and $\beta$ satisfying the condition $0 \leq \alpha \leq \beta$, as follow:\\
For every continuous function $f$ on $[0,1]$ and for every $n \in \mathbb{N}$ the operator is denoted by $B_n^{\alpha, \beta}$ and is given by 
\begin{equation}\label{Sta}
B_n^{\alpha,\beta}(f;x)=\sum_{k=0}^{n} b_{n,k}(x)f\left(\dfrac{k+\alpha}{n+\beta} \right),
\end{equation}
where $b_{n,k}$'s are given by (\ref{BerBas}). The operators given by (\ref{Sta}) are called the Bernstein-Stancu operators (or polynomials)\cite{Alt}. They are linear positive operators and generalization of the Bernstein operators in the sense that for $\alpha = 0, \beta = 0$, the Bernstein-Stancu operators reduce to Bernstein operators. A vast literature is available about the Bernstein polynomial, its different generalizations, modifications and graphical representations of them. We refer the readers to \cite{Buy}-\cite{Moh1}, \cite{RBG1}-\cite{VNM2} and references therein. \\
The Bernstein-Stancu operators use the equidistant nodes $x_0 = \dfrac{\alpha}{n+\beta}, x_1 = x_0+h, \ldots x_n = x_0 + nh$, where $h = \dfrac{1}{n+\beta}$. As $B_n^{\alpha,\beta}(f;0)=f\left(\dfrac{\alpha}{n+\beta}\right)$ and $B_n^{\alpha,\beta}(f;1)=f\left(\dfrac{n+\alpha}{n+\beta}\right)$, $B_n^{\alpha,\beta}(f; \cdot)$ interpolates function $f$ in $x=0$ if $\alpha=0$ and in $x=1$ if $\alpha=\beta$. \\
For the test functions $e_i(t)=t^i$, where $t \in [0,1]$ and $i=0,1,2$; we have
\begin{equation}\label{Test0}
B_n^{\alpha,\beta}(e_0;x)=1,
\end{equation}
\begin{equation}\label{Test1}
B_n^{\alpha,\beta}(e_1;x)=x+\dfrac{\alpha-\beta x}{n + \beta},
\end{equation}
\begin{eqnarray}\label{Test2}
 B_n^{\alpha,\beta}(e_2;x) = x^2 +  \dfrac{nx(1-x) + (\alpha-\beta x)(2nx + \beta x + \alpha)}{(n + \beta)^2},
\end{eqnarray}
\noindent As an immediate consequence of (\ref{Test0}), (\ref{Test1}), (\ref{Test2}) and from the Popoviciu-Bohman-Korovkin theorem\cite{Alt}, we can state that for any $f \in C([0,1])$ the sequence $\left(B_n^{\alpha,\beta}(f;\cdot)\right)_{n \in \mathbb{N}}$ converges uniformly to $f$.
\\
\noindent For given $f \in C([0,1])$, $x_1, x_2 \in [0,1]$ and $\delta > 0$,

\begin{equation}\label{Mod}
\omega(f;\delta)= \max_{\left|x_1 - x_2\right| \leq \delta} \left\lbrace\left|f(x_1) - f(x_2)\right|\right\rbrace,
\end{equation}
denotes the modulus of continuity of $f$.
\section{Main Results}
The nodes, for fixed $n \in \mathbb{N}$, for the Bernstein operators (\ref{Ber}) are 
\begin{equation}\label{Node1}
0=\dfrac{0}{n},\dfrac{1}{n}, \dfrac{2}{n}, \ldots, \dfrac{n-1}{n}, \dfrac{n}{n}=1,
\end{equation}
while for the Bernstein-Stancu operators (\ref{Sta}), they are 
\begin{equation}\label{Node2}
\dfrac{\alpha}{n + \beta},\dfrac{1 + \alpha}{n + \beta}, \dfrac{2 + \alpha}{n + \beta}, \ldots, \dfrac{n + \alpha}{n + \beta}.
\end{equation} 
For fixed $n \in \mathbb{N}$, the nodes given by (\ref{Node1}) will be denoted by $\left(\dfrac{k}{n}\right)$ and the nodes given by  (\ref{Node2}) will be denoted by $\left(\dfrac{k+\alpha}{n+\beta}\right)$, where $k=0,1,\ldots,n$; $\alpha, \beta$ are two non-negative numbers with $0 \leq \alpha \leq \beta$.
\begin{theorem}\label{T1}
The nodes $\left(\dfrac{k+\alpha}{n+\beta}\right)$ are distributed as evenly as the nodes $\left(\dfrac{k}{n}\right)$ when $n \rightarrow \infty$, for any $0 \leq \alpha \leq \beta$.
\end{theorem}
\begin{proof}
For $k=0$, $\lim_{n \rightarrow \infty} \dfrac{k+\alpha}{n+\beta}=0=\lim_{n \rightarrow \infty}\dfrac{k}{n}$. \\ \indent  \indent
For $k=n$, $\lim_{n \rightarrow \infty} \dfrac{n+\alpha}{n+\beta}=1=\lim_{n \rightarrow \infty}\dfrac{n}{n}$. \\
Also,
\begin{equation}\label{E1}
\dfrac{k + \alpha}{n + \beta} - \dfrac{k}{n} = \dfrac{n\alpha - k\beta}{n(n+\beta)},
\end{equation}
and $\dfrac{\left|n\alpha - k\beta\right|}{n(n+\beta)}\leq \dfrac{\alpha + \beta}{n+\beta}\rightarrow 0$ as $n \rightarrow \infty$, which proves the theorem.
\end{proof}
\noindent As a consequence of Theorem \ref{T1}, we have an alternate proof of uniform convergence of $\left(B_n^{\alpha,\beta}(f;\cdot)\right)_{n \in \mathbb{N}}$ to $f \in C([0,1])$. 
\begin{cor}\label{C1}
For $f \in C([0,1])$, the sequence $\left(B_n^{\alpha,\beta}(f;\cdot)\right)_{n \in \mathbb{N}}$ converges uniformly to $f$ in $[0,1]$.
\end{cor}
\begin{proof}
We observe that 
\begin{eqnarray*}
&&\left|B_n^{\alpha,\beta}(f;x)-B_n(f;x)\right| \\ && = \left|\sum_{k=0}^n b_{n,k}(x) \left[f\left(\dfrac{k + \alpha}{n + \beta} \right) - f\left(\dfrac{k}{n}\right) \right]  \right| \\ && \leq  \sum_{k=0}^n b_{n,k}(x) \left|f\left(\dfrac{k + \alpha}{n + \beta} \right) - f\left(\dfrac{k}{n}\right)\right|.
\end{eqnarray*}
Using uniform continuity of $f$, (\ref{E1}) and $\sum_{k=0}^n b_{n,k}(x) = 1$, we have,
\begin{eqnarray*}
\left|B_n^{\alpha,\beta}(f;x)-B_n(f;x)\right| \rightarrow 0, \quad n \rightarrow \infty.
\end{eqnarray*}
Now, using the uniform convergence of $B_n(f;\cdot)$ to $f \in C([0,1])$ (refer \cite{Alt}), we get
\begin{eqnarray*}
\left|B_n^{\alpha,\beta}(f;x)-f(x)\right|  \leq  \left|B_n^{\alpha,\beta}(f;x)-B_n(f;x)\right|  + \left|B_n(f;x)-f(x)\right| \rightarrow 0, 
\end{eqnarray*}
as $n \rightarrow \infty$.
\end{proof}
\begin{figure}[h!]
\centering
\includegraphics[width=12cm,height=4.25cm]{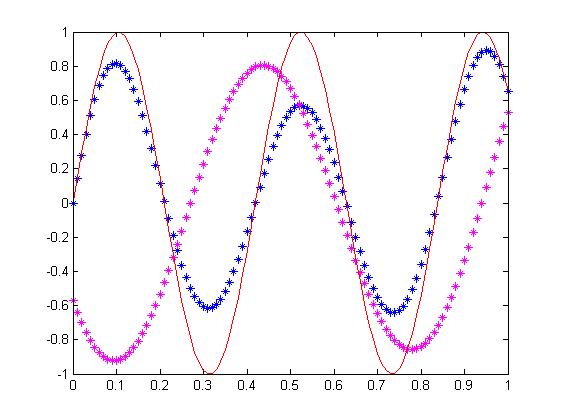}
	\caption{\scriptsize \sl Graphs of $f(x) = \sin (15x)$ ("Red"), Bernstein polynomial of $f$ ("Blue") and Bernstein-Stancu polynomial of $f$ with $\alpha=20$, $\beta=30$ ("Magenta") for $n = 50$ on $[0,1]$}
\label{F1}
\end{figure}
\begin{figure}[h!]
\centering
\includegraphics[width=12cm,height=4.25cm]{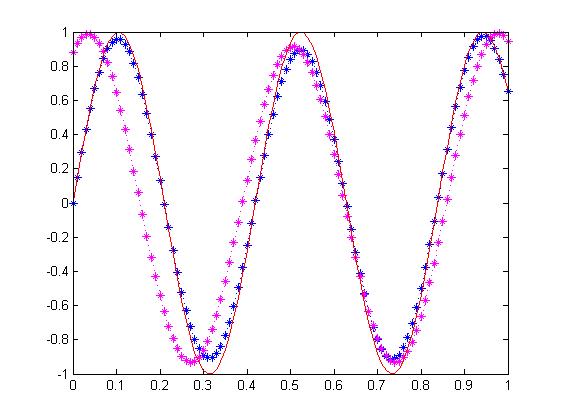}
	\caption{\scriptsize \sl Graphs of $f(x) = \sin (15x)$ ("Red"), Bernstein polynomial of $f$ ("Blue") and Bernstein-Stancu polynomial of $f$ with $\alpha=20$, $\beta=30$ ("Magenta") for $n = 250$ on $[0,1]$}
\label{F2}
\end{figure}
\noindent In figures \ref{F1} and \ref{F2}, function $f(x) = \sin (15x) $, its Bernstein polynomials and Bernstein-Stancu polynomials of degrees $n= 50$ and $n=250$, respectively, are plotted over $[0,1]$. It can be seen that for the higher value of $n$, the convergence is better.
\\
\noindent In the following corollary an upper estimate for the Bernstein-Stancu operators is derived.
\begin{cor}\label{C2}
For $f \in C([0,1])$ and $x \in [0,1]$, 
\begin{eqnarray*}
\vert B_n^{\alpha,\beta}(f;x)-f(x)\vert  \leq  c \cdot \omega(f ; n^{-1/2}),
\end{eqnarray*}
where $\omega(f;\delta)$ is the modulus of continuity of $f$ given by (\ref{Mod}) and $c$ is a constant depending upon $\alpha, \beta$.
\end{cor}
\begin{proof}
We observe that
\begin{eqnarray*}
\left|B_n^{\alpha,\beta}(f;x)-f(x)\right| & \leq & \left|B_n^{\alpha,\beta}(f;x)-B_n(f;x)\right| + \left|B_n(f;x)-f(x)\right| \\ & \leq & \sum_{k=0}^n b_{n,k}(x)\left|f\left(\dfrac{k + \alpha}{n + \beta} \right) - f\left(\dfrac{k}{n} \right)\right|  + \left|B_n(f;x)-f(x)\right|.
\end{eqnarray*}
Now, from (\ref{E1})
\begin{eqnarray*}
\left|\dfrac{k + \alpha}{n + \beta}  - \dfrac{k}{n} \right| & \leq & \dfrac{\left|n\alpha - k\beta\right|}{n(n+\beta)}\leq \dfrac{\alpha + \beta}{n+\beta}
\end{eqnarray*}
for $k=0,1,\ldots,n$, therefore
\begin{eqnarray*}
\left|B_n^{\alpha,\beta}(f;x)-f(x)\right|  & \leq & \omega\left(f;\dfrac{\alpha+\beta}{n + \beta}\right) + c_1 \cdot \omega(f ; n^{-1/2})\\ & \leq & c \cdot \omega(f ; n^{-1/2}),
\end{eqnarray*}
where $\left|B_n(f;x)-f(x)\right| \leq c_1 \cdot \omega(f ; n^{-1/2})$ (\cite{Radu}-\cite{Fin}), $c_1$ is an absolute constant and $c$ is a constant depending upon $\alpha, \beta$.\\
\end{proof}
\begin{theorem}\label{T2}
The nodes $\left(\dfrac{k+\alpha}{n+\beta}\right)$ are more closely clustered around $\dfrac{\alpha}{\beta}$, $(\beta > 0)$ then the nodes $\left(\dfrac{k}{n}\right)$.  In fact, we have $ \left| \dfrac{k + \alpha}{n + \beta} - \dfrac{\alpha}{\beta}\right| \leq \left|\dfrac{k}{n} - \dfrac{\alpha}{\beta}\right|$.
\end{theorem}
\begin{proof}
From (\ref{E1}), if $\dfrac{k}{n} > \dfrac{\alpha}{\beta}$, then  $n\alpha-k\beta < 0$, hence $\dfrac{k+\alpha}{n+\beta}-\dfrac{k}{n} < 0$, that is, $\dfrac{k+\alpha}{n+\beta} < \dfrac{k}{n}$. \\
Also, we have 
\begin{equation}\label{E2}
\dfrac{k + \alpha}{n + \beta} - \dfrac{\alpha}{\beta} = \dfrac{n}{n+\beta}\left(\dfrac{k}{n}-\dfrac{\alpha}{\beta} \right).
\end{equation}
So, $\dfrac{k}{n} > \dfrac{\alpha}{\beta}$ implies $\dfrac{k + \alpha}{n + \beta} > \dfrac{\alpha}{\beta}$. \\
Similarly, again using (\ref{E1}) and (\ref{E2}), if $\dfrac{k}{n} < \dfrac{\alpha}{\beta}$, then $\dfrac{k+\alpha}{n+\beta} > \dfrac{k}{n}$ and $\dfrac{k + \alpha}{n + \beta} < \dfrac{\alpha}{\beta}$. \\
Thus, for $\dfrac{k}{n} > \dfrac{\alpha}{\beta}$, we have $ \dfrac{\alpha}{\beta} < \dfrac{k + \alpha}{n + \beta
} < \dfrac{k}{n}$ and for $\dfrac{k}{n} < \dfrac{\alpha}{\beta}$, we have $\dfrac{k}{n} < \dfrac{k + \alpha}{n + \beta}<\dfrac{\alpha}{\beta}$, which concludes to $ \left| \dfrac{k + \alpha}{n + \beta} - \dfrac{\alpha}{\beta}\right| \leq \left|\dfrac{k}{n} - \dfrac{\alpha}{\beta}\right|$.
\end{proof}
\noindent Remark: Actually the nodes $\left(\dfrac{k}{n}\right)$ cross the nodes $\left(\dfrac{k+\alpha}{n+\beta}\right)$ at $\dfrac{\alpha}{\beta}$. These can be observed in the figures \ref{F3}, \ref{F4} and \ref{F5} for $\alpha =17 , \beta = 100$, $\alpha =47 , \beta = 100$ and $\alpha =77 , \beta = 100$, respectively with fixed $n = 100$. Clustering of the nodes $\left(\dfrac{k+\alpha}{n+\beta}\right)$ about $\dfrac{\alpha}{\beta}$ is also seen clearly.
\begin{figure}[h!]
\centering
\includegraphics[width=12cm,height=4.25cm]{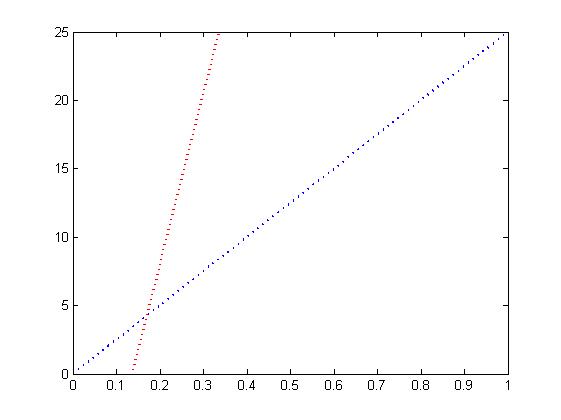}
	\caption{\scriptsize \sl Nodes clustering of Bernstein polynomial ("Blue") and Bernstein-Stancu polynomial ("Red") for $\alpha = 17, \beta = 100$, $n = 25$ on $[0,1]$.}
\label{F3}
\end{figure}
\begin{figure}[h!]
\centering
\includegraphics[width=12cm,height=4.25cm]{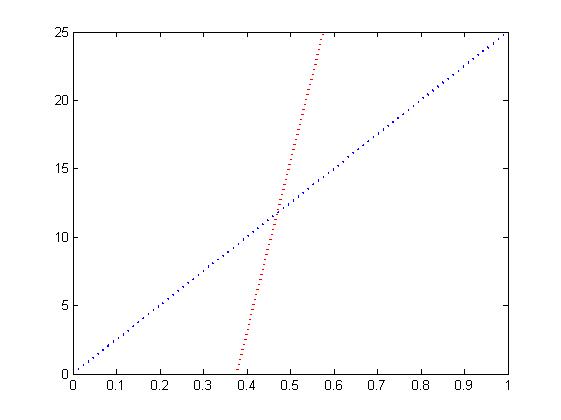}
	\caption{\scriptsize \sl Nodes clustering of Bernstein polynomial ("Blue") and Bernstein-Stancu polynomial ("Red") for $\alpha = 47, \beta = 100$, $n = 25$ on $[0,1]$.}
\label{F4}
\end{figure}
\begin{figure}[h!]
\centering
\includegraphics[width=12cm,height=4.25cm]{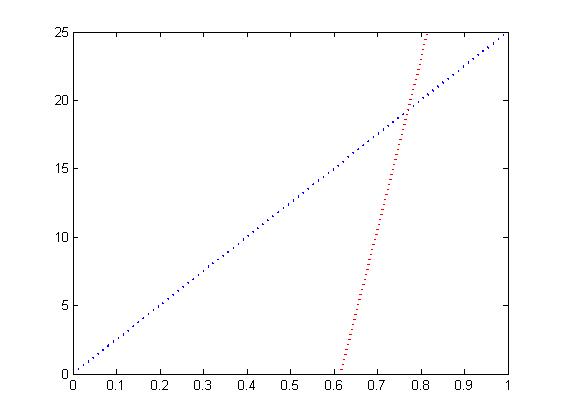}
	\caption{\scriptsize \sl Nodes clustering of Bernstein polynomial ("Blue") and Bernstein-Stancu polynomial ("Red") for $\alpha = 77, \beta = 100$, $n = 25$ on $[0,1]$.}
\label{F5}
\end{figure}
\begin{figure}[h!]
\centering
\includegraphics[width=12cm,height=4.25cm]{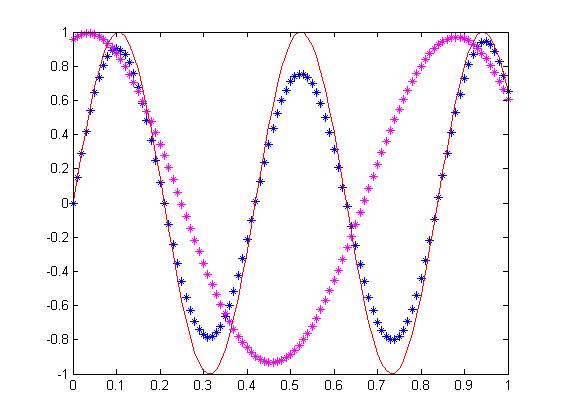}
	\caption{\scriptsize \sl Comparison of graphs of $f(x) = \sin (15x)$ ("Red") with Bernstein polynomial of $f$ ("Blue") and Bernstein-Stancu polynomial of $f$ ("Magenta") with $\alpha = 17$, $\beta = 100$ and $n = 100$.}
\label{F6}
\end{figure}
\begin{figure}[h!]
\centering
\includegraphics[width=12cm,height=4.25cm]{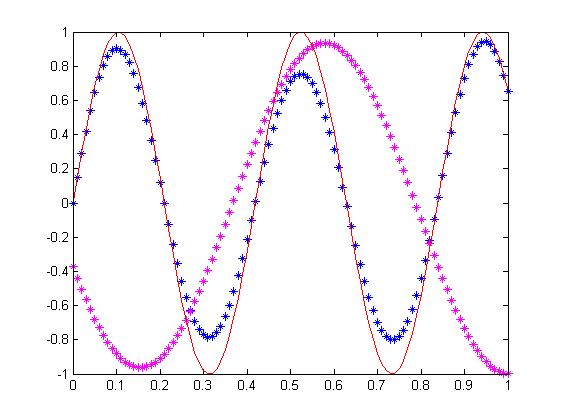}
	\caption{\scriptsize \sl Comparison of graphs of $f(x) = \sin (15x)$ ("Red") with Bernstein polynomial of $f$ ("Blue") and Bernstein-Stancu polynomial of $f$ ("Magenta") with $\alpha = 47$, $\beta = 100$ and $n = 100$.}
\label{F7}
\end{figure}
\begin{figure}[h!]
\centering
\includegraphics[width=12cm,height=4.25cm]{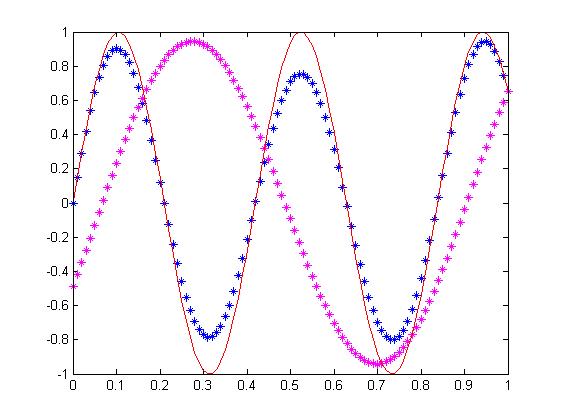}
	\caption{\scriptsize \sl Comparison of graphs of $f(x) = \sin (15x)$ ("Red") with Bernstein polynomial of $f$ ("Blue") and Bernstein-Stancu polynomial of $f$ ("Magenta") with $\alpha = 77$, $\beta = 100$ and $n = 100$.}
\label{F8}
\end{figure}
\\
\noindent Remark: In figures \ref{F6}, \ref{F7} and \ref{F8}, the graphs of function $f(x) = \sin (15x)$, its Bernstein polynomial and Bernstein-Stancu polynomial of degree $n = 100$ for the nodes plotted in figures \ref{F3}, \ref{F4} and \ref{F5}, respectively are plotted. It can be conjectured that Bernstein-Stancu polynomials can be used to get better approximation at a point to the Bernstein polynomials with an appropriate choice of parameters $\alpha, \beta$ reducing the number of terms used in calculation, which is proved in Theorem \ref{T4}.
\newpage
\begin{theorem}\label{T3}
Suppose that $\alpha_1, \alpha_2, \beta_1,$ and $\beta_2$ are non-negative numbers such that $0 \leq \alpha_1 \leq \alpha_2$, $0 \leq \beta_1 \leq \beta_2$ and $\dfrac{\alpha_1}{\beta_1} = \dfrac{\alpha_2}{\beta_2} = m$(say), then the nodes $\left(\dfrac{k+\alpha_2}{n+\beta_2}\right)$ will be more closely clustered around $m$ then $\left(\dfrac{k+\alpha_1}{n+\beta_1}\right)$.
\end{theorem}
\begin{proof}
We observe that
\begin{eqnarray}\label{E3}
\dfrac{k + \alpha_1}{n+\beta_1} - \dfrac{k + \alpha_2}{n+\beta_2} = \dfrac{n(\beta_2 - \beta_1)(k/n - m)}{(n+\beta_1)(n+\beta_2)},
\end{eqnarray}
\begin{equation}\label{E4}
\dfrac{k + \alpha_1}{n+\beta_1}-m = \dfrac{n}{n + \beta_1}\left( \dfrac{k}{n} - m\right),
\end{equation}
and
\begin{equation}\label{E5}
\dfrac{k + \alpha_2}{n+\beta_2}-m = \dfrac{n}{n + \beta_2}\left( \dfrac{k}{n} - m\right).
\end{equation}
From (\ref{E3}), (\ref{E4}) and (\ref{E5}), for $\dfrac{k}{n} < m$, we have $\dfrac{k + \alpha_1}{n+\beta_1} < \dfrac{k + \alpha_2}{n+\beta_2} < m$ and for $\dfrac{k}{n} > m$, we have $\dfrac{k + \alpha_1}{n+\beta_1} > \dfrac{k + \alpha_2}{n+\beta_2} > m$.
\end{proof}
\noindent In figure \ref{F9}, we can see the clustering phenomenon of nodes with respect to the choices of parameters $\alpha,\beta$.  
\begin{figure}[h!]
\centering
\includegraphics[width=12cm,height=4.5cm]{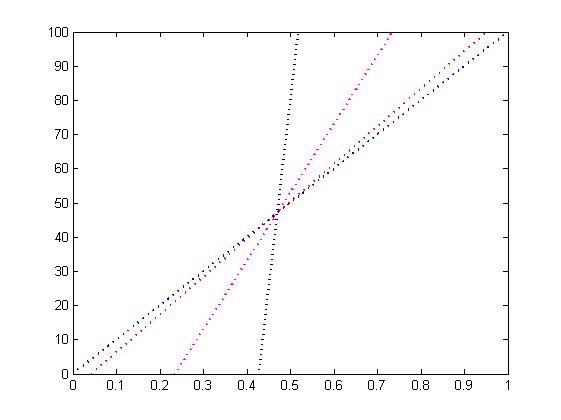}
	\caption{\scriptsize \sl Clustering of nodes of Bernstein polynomial of $f$ ("Blue") and Bernstein-Stancu polynomial of $f$ for $\alpha=4.7$, $\beta=10$ ("Red"); $\alpha=47$, $\beta=100$ ("Magenta"); $\alpha=470$, $\beta=1000$ ("Black") with fixed $n = 100$ and $m = 0.47$.}
\label{F9}
\end{figure}
\begin{figure}[h!]
\centering
\includegraphics[width=12cm,height=4.5cm]{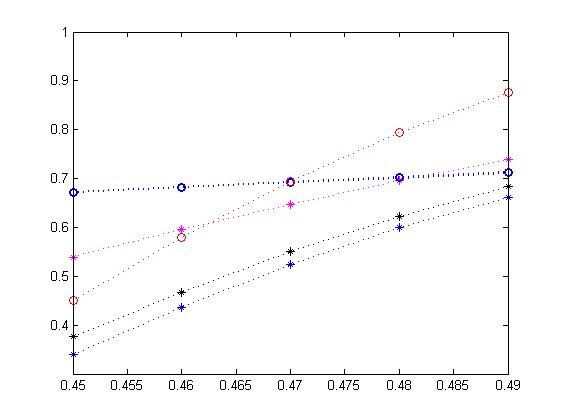}
	\caption{\scriptsize \sl Comparison of graphs of $f(x) = \sin (15x)$ ("Red o") with Bernstein polynomial of $f$ ("Blue *") and Bernstein-Stancu polynomials of $f$ for $\alpha=4.7$, $\beta=10$ ("Black *"); $\alpha=47$, $\beta=100$ ("Magenta *"); $\alpha=470$, $\beta=1000$ ("Blue o") with fixed $n = 100$ and $m = 0.47$.}
\label{F10}
\end{figure}
\\Figure \ref{F10} includes the graphs of the function $f(x)= \sin(15x)$, its Bernstein polynomials and Bernstein-Stancu polynomials for different values of parameters $\alpha, \beta$ with $m = 0.47$ of degree $n=100$. Observe the graphs near $m = 0.47$.
\begin{theorem}\label{T4}
Let $0 < \alpha < \beta$. Let $\left(\alpha_j\right)$ and $\left(\beta_j\right)$ be the increasing sequence of real numbers with $0 < \alpha_j < \beta_j$, $\alpha_j \rightarrow \infty$, as $j \rightarrow \infty$ and $\dfrac{\alpha_j}{\beta_j} = \dfrac{\alpha}{\beta}$ for all $j = 1, 2, 3, \ldots$. For a function $f$ defined on $[0,1]$ and continuous at $\dfrac{\alpha}{\beta}$ and for any fix natural number $n$, $B_n^{\alpha_l, \beta_l}$ can be brought arbitrarily closed to $f\left( \dfrac{\alpha}{\beta} \right)$, for some $0 < \alpha_l < \beta_l$.
\end{theorem}
\begin{proof}
Continuity of $f$ at $\dfrac{\alpha}{\beta}$ implies that for any $\epsilon > 0$, there exists $\delta > 0$, such that for any $x \in [0,1]$ with $\left|x - \dfrac{\alpha}{\beta} \right| < \delta$, we have $\left|f(x) -f\left(\dfrac{\alpha}{\beta}\right) \right| < \epsilon$. From the definition of the sequences $\left(\alpha_j\right)$ and $\left(\beta_j\right)$,  for a fixed natural number $n$ and $k = 0,1, \ldots , n$, we can find some $0 < \alpha_l < \beta_l$ such that $\left|\dfrac{k+\alpha_l}{n+\beta_l} - \dfrac{\alpha}{\beta} \right| < \delta$. Indeed, as
\begin{eqnarray*}
	\left|\dfrac{k+\alpha_l}{n+\beta_l} - \dfrac{\alpha}{\beta} \right| &=& \left|\dfrac{k\beta+\alpha_l\beta - n\alpha-\beta_l \alpha}{\beta(n+\beta_l)}\right| \\ &=& \left|\dfrac{n\left(k/n - \alpha/\beta\right)} {n+\beta_l}\right|\\& \leq &\dfrac{2n} {n+\beta_l}  \rightarrow 0
\end{eqnarray*}
when $\beta_l \rightarrow \infty$. So, 
\begin{eqnarray*}
\left| B_n^{\alpha_l, \beta_l}(f;x) - f\left(\dfrac{\alpha}{\beta}\right) \right| & =& \left|\sum_{k=0}^n b_{n,k}(x)\left[f\left(\dfrac{k + \alpha_l}{n+ \beta_l} \right) - f\left(\dfrac{\alpha}{\beta} \right) \right]\right| \\ & \leq&   \sum_{k=0}^n b_{n,k}(x)\left|f\left(\dfrac{k + \alpha_l}{n+ \beta_l} \right) - f\left(\dfrac{\alpha}{\beta} \right) \right| \\ &< &\epsilon.
\end{eqnarray*}
\end{proof}
\section{Concluding Remark} The sensitivity of parameters $\alpha, \beta$ appearing in the Bernstein-Stancu operators is observed. Graphical representations are aimed at visualization of the study better way. The approximation of the Bernstein-Stancu operators can be made based up on the choice parameters and continuity at $\dfrac{\alpha}{\beta}$ (Theorem \ref{T4}). A better approximation of a continuous function $f$ at any $x$ in $[0,1]$ can be effectively obtained using the Bernstein-Stancu operators with appropriate choice of parameters $\alpha, \beta$ to the Bernstein operators, if one wants to limit calculation to less number of terms.
\section*{Acknowledgements}
The authors are thankful to Rodin Lusan and Barbara Strazzabosco of Zentralblatt MATH, Berlin, Germany, who took all the trouble to manage us the paper \cite{Sta} from Studia UBB, Romania.

\end{document}